\documentclass[12pt, a4paper]{article}
\usepackage{amsfonts}
\usepackage{amsmath}
\usepackage[latin1]{inputenc}
\usepackage{amsthm}
\usepackage{pdfsync}
\usepackage{amssymb}
\usepackage[francais,english]{babel}

\numberwithin{equation}{section}
\theoremstyle{plain}
\newtheorem{thm}{Theorem}[section]
\newtheorem{coro}[thm]{Corollary}
\newtheorem{prop}[thm]{Proposition}
\newtheorem{lem}[thm]{Lemma}
\newtheorem{defi}[thm]{Definition}
\theoremstyle{definition}

\theoremstyle{remark}
\newtheorem{rem}[thm]{Remark}

\newcommand\E{{\mathbb E}}

\newcommand{\R}{\mathbb{R}}

\newcommand\pref[1]{(\ref{#1})}

\let \eps\varepsilon

\newcommand\tr{\mathop{\mathrm{Tr}}\nolimits}

\def\<#1,#2>{\left<#1,#2\right>}

\newcommand\PP{{\mathbb P}}

\begin{document}

\title{Exponential convergence for a convexifying equation and a non-autonomous gradient flow for global minimization}
\author{G. Carlier \thanks{{\scriptsize CEREMADE, UMR CNRS 7534,
Universit\'e Paris IX Dauphine, Pl. de Lattre de Tassigny, 75775 Paris Cedex
16, France. \texttt{carlier@ceremade.dauphine.fr}}}, A. Galichon \thanks{
{\scriptsize D\'epartement d'Economie, Ecole polytechnique, 91128 Palaiseau cedex. \texttt{\
alfred.galichon@polytechnique.edu}. Galichon gratefully acknowledges support from Chaire EDF-Calyon \textquotedblleft Finance and D\'{e}veloppement Durable\textquotedblright\ and FiME, Laboratoire de Finance des March\'es de l'Energie (www.fime-lab.org),  and Chaire Axa \textquotedblleft Assurance et
Risques Majeurs\textquotedblright.  
}}}
\maketitle

\begin{abstract}
We consider an evolution equation similar to that introduced by Vese in \cite%
{vese} and whose solution converges in large time to the convex envelope of
the initial datum. We give a stochastic control representation for the
solution from which we deduce, under quite general assumptions that the
convergence in the Lipschitz norm is in fact exponential in time. We then
introduce a non-autonomous gradient flow and prove that its trajectories all
converge to minimizers of the convex envelope.
\end{abstract}



\section{Introduction}

In an interesting paper \cite{vese}, L.Vese considered the following PDE:
\begin{equation}  \label{vese0}
\partial_t u= \sqrt{1+ \vert \nabla u \vert^2} \min(0, \lambda_1(D^2 u)), \;
u\vert_{t=0}= u_0
\end{equation}
where $\lambda_1(D^2 u)$ denotes the smallest eigenvalue of the Hessian
matrix $D^2 u$. Vese proved, under quite general assumptions on the initial
condition $u_0$, that the viscosity solution of (\ref{vese0}) converges as $%
t\to \infty$ to $u_0^{**}$ the convex envelope of $u_0$. Starting from this
result, Vese developed an original and purely PDE approach to approximate
convex envelopes (which is in general a delicate problem as soon as the
space dimension is larger than $2$). More recently, A. Oberman \cite%
{oberman1}, \cite{oberman2}, \cite{os}, noticed that the convex envelope can
be directly characterized via a nonlinear elliptic PDE of obstacle type and
developed this idea for numerical computation of convex envelopes as well.
As noticed by Oberman, the solution of the PDE he introduced naturally has a
stochastic control representation. This is of course also the case for the
evolutionary equation of L.Vese and, as we shall see, this representation
will turn out to be very useful to obtain convergence estimates.

\smallskip

In the present paper, we will focus on an evolution equation similar to (\ref%
{vese0}) and will study some of its properties thanks to the stochastic
control representation of the solution. Under natural assumptions on the
initial datum, our first main result is that $u(t,.)$ converges to the
convex envelope of $u_0$ exponentially fast in the Lipschitz norm. From this
convergence, we deduce that trajectories of the non-autonomous gradient flow
\begin{equation*}
\dot x(t)=-\nabla u(t,x(t))
\end{equation*}
all converge to a minimizer of the convex envelope. 

\smallskip

The paper is organized as follows. In section \ref{sec1}, we introduce the
\emph{convexifying evolution equation} and recall some basic facts about
convex envelopes. In section \ref{scr}, we give a stochastic control
representation for the solution of the convexifying evolution equation.
Section \ref{regpu} gives some regularity properties of the solution. Our
exponential convergence result is then proved in section \ref{expco} by
simple probabilistic arguments. Finally, convergence of the trajectories of
the non-autonomous gradient flow are proved in section \ref{nagf}.

\section{A convexifying evolution equation}

\label{sec1}

In the present paper, we will consider a slight variant of (\ref{vese0}),
namely:
\begin{equation}  \label{cee}
\partial_t u(t,x)= \min(0, \lambda_1(D^2 u(t,x))), \; (t,x)\in
(0;\infty)\times \mathbb{R}^d, \; u\vert_{t=0}= u_0
\end{equation}
In the sequel, we shall refer to (\ref{cee}) as the \emph{convexifying
evolution equation}. Following the same arguments of the proof of Vese \cite%
{vese} (also see remark \ref{rem1} below), one can prove under mild
assumptions on $u_0$ that the solution converges pointwise to the convex
envelope $u_0^{**}$ of the initial condition. Our aim will be to quantify
this convergence and this goal will be achieved rather easily by using a
stochastic representation formula for the solution of (\ref{cee}). Before we
do so, let us recall some basic facts about the convex envelope.

Given a continuous (say) and bounded from below function $u_0$ defined on $%
\mathbb{R}^d$, the convex envelope of $u_0^{**}$ is the largest convex
function that is everywhere below $u_0$. The convex envelope is a very
natural object in many contexts and in particular in optimization since $u_0$
and $u_0^{**}$ have the same infimum but $u_0^{**}$ is in principle much
simpler to minimize since it is convex. One can also define $u_0^{**}$ as
the supremum of all affine functions that are below $u_0$ and thus define $%
u_0^{**}$ as the ``Legendre Transform of the Legendre Transform'' of $u_0$
(and this is where the notation ``$^{**}$'' comes from). Rather than iterating
the Legendre transform, let us recall the well-known formula:
\begin{equation}  \label{envconvex}
u_0^{**}(x)=\inf\left\{ \sum_{i=1}^{d+1} \lambda_i u_0(x_i) \; : \;
\lambda_i\geq 0, \; \sum_{i=1}^{d+1} \lambda_i=1, \;
\sum_{i=1}^{d+1}\lambda_i x_i=x\right\}, \; \forall x\in \mathbb{R}^d
\end{equation}
(the fact that one can restrict to $d+1$ points follows from
Carath\'eodory's theorem) which can also be written in probalistic terms as
\begin{equation}  \label{prob}
u_0^{**}(x)=\inf \Big\{ {\mathbb{E}}(u_0(x+X)) \; : \: {\mathbb{E}}(X)=0%
\Big\}.
\end{equation}
The latter formula strongly suggests that a good approximation for the
convex envelope should be
\begin{equation}  \label{markov}
u(t,x):=\inf_{\sigma \; : \; \vert \sigma \vert \leq 1} \Big\{ {\mathbb{E}}%
\Big(u_0(x+\int_0^t \sqrt{2} \sigma_s dW_s)\Big)\Big\}
\end{equation}
for large $t$ where $(W_s)_{s\geq 0}$ is a standard Brownian motion, $%
\sigma_s$ is a $d\times d$-matrix valued process that is adapted to the
Brownian filtration and $\vert \sigma\vert$ stands for the matrix norm $%
\vert \sigma\vert:=\sqrt{\mathop{\mathrm{Tr}}\nolimits(\sigma \sigma^T)}$.

\smallskip

In order to keep things as elementary as possible, from now on, we shall
always assume that $u_0$ satisfies:
\begin{equation}  \label{hypsuru0}
u_0\in C^{1,1}(\mathbb{R}^d), \; \lim_{\vert x\vert \to \infty} \frac{u_0(x)%
} {\vert x \vert}=+\infty, \; \exists R_0>0 \mbox{ : } u_0=u_0^{**}
\mbox{
outside $\overline{B}_{R_0}$}.
\end{equation}
The coercivity assumption guarantees that the infimum in formula (\ref%
{envconvex}) is actually achieved. The assumption that $u_0$ is $C^{1,1}$
implies that so is $u_0^{**}$ (see \cite{Kir}) and we will see that it also
implies that $u(t,.)$ remains $C^{1,1}$. Finally, the assumption that $u_0$
and $u_0^{**}$ agree outside of some ball, eventhough not as essential as
the previous ones, will be convenient and allow us to work mainly on a ball
instead of on the whole space.

\section{Stochastic control representation}

\label{scr}

As we shall see (but this should already be clear to stochastic
control-oriented readers), the value function of (\ref{markov}) is in fact
characterized by the PDE:
\begin{equation}  \label{edp}
\partial_t v=\min(0, \lambda_1(D^2v))
\end{equation}
in the viscosity sense that we now recall (for the sake of simplicity, we
will restrict ourselves to the framework of continuous solutions which is
sufficient in our context):

\begin{defi}
\label{defvisc} Let $\Omega$ be some open subset of $\mathbb{R}^d$ and let $%
v $ be continuous on $(0,+\infty)\times \Omega$, then $v$ is :

\begin{itemize}
\item a viscosity subsolution of (\ref{edp}) on $(0, +\infty)\times \Omega$
if for every smooth function $\varphi\in C^2((0,+\infty)\times \Omega)$ and
every $(t_0, x_0)\in (0, +\infty)\times \Omega$ such that $%
(u-\varphi)(t_0,x_0)=\max_{(0, +\infty)\times \Omega} (u-\varphi)$ one has
\begin{equation*}
\partial_t \varphi(t_0, x_0)\leq \min(0, \lambda_1(D^2 \varphi(t_0,x_0))),
\end{equation*}

\item a viscosity supersolution of (\ref{edp}) on $(0, +\infty)\times \Omega$
if for every smooth function $\varphi\in C^2((0,+\infty)\times \Omega)$ and
every $(t_0, x_0)\in (0, +\infty)\times \Omega$ such that $%
(u-\varphi)(t_0,x_0)=\min_{(0, +\infty)\times \Omega} (u-\varphi)$ one has
\begin{equation*}
\partial_t \varphi(t_0, x_0)\geq \min(0, \lambda_1(D^2 \varphi(t_0,x_0))),
\end{equation*}

\item a viscosity subsolution of (\ref{edp}) on $(0, +\infty)\times \Omega$
if it is both a viscosity subsolution and a viscosity supersolution.
\end{itemize}
\end{defi}

We then have the following stochastic representation formula for (\ref{cee}):

\begin{thm}
\label{stochrep} There is a unique continuous function $u$ on $[0,+\infty)\times \R^d$ that agrees with $u_0$ at $t=0$, that is a viscosity solution of (\ref%
{cee}) and that agrees with $u_0^{**}$ outside $B_{R_0}$. It admits the
following representation
\begin{equation}  \label{repformula}
u(t,x)=\inf_{\sigma \; : \; \vert \sigma \vert \leq 1} \Big\{ {\mathbb{E}}%
\Big(u_0(x+\int_0^t \sqrt{2} \sigma_s dW_s)\Big) \Big\}, \; t\geq 0, \; x\in
\mathbb{R}^d
\end{equation}
where $(W_s)_{s\geq 0}$ is a standard Brownian motion and $\vert \sigma\vert$
stands for the matrix norm $\vert \sigma\vert:=\sqrt{\mathop{\mathrm{Tr}}%
\nolimits(\sigma \sigma^T)}$.
\end{thm}

\begin{proof}
Recalling that for every symmetric matrix $S$ one has
\[\min(0, \lambda_1(S))=\min_{\vert \sigma\vert\leq 1} \text{Tr} (\sigma \sigma^T S),\]
the fact that formula \pref{repformula} actually defines a viscosity solution is a classical fact from stochastic control theory  (see for instance \cite{fs} or \cite{touzi}) and uniqueness  follows from well-known comparison principles (e.g Theorem 4.1 in \cite{cgg}). Continuity (Lipschitz continuity in fact) of the value function $u$ will be established in section \ref{regpu}.  

\end{proof}

\begin{rem}\label{kentarem}
\textbf{Optimal feedback control.} 
Very formally, if the solution $u$ of the PDE were very well-behaved then, as usual in control theory, one could find an optimal feedback (Markov) control depending on $D^2 u$ (since there is no drift). Introduce a time-dependent vector field $Z=Z(t,y)$, as follows. If $\lambda
_{1}( D^{2}u(t, y)) <0$, then let $Z(t, y)$ be a unit eigenvector associated to $\lambda_{1}( D^{2} u(t,y))$  and let $Z(t,y)=0$ otherwise. So that in any case:
\[ \tr(\sigma(t,y) \sigma(t,y)^T D^2 u(t,y))= \min(0, \lambda_1(D^2 u(t, y)), \mbox{ and } \vert \sigma(t,y)\vert\leq 1, \]
where $\sigma$ is the projector
 \[ \sigma(t,y):=Z(t,y)\otimes Z(t,y).\]

Of course the problem is that $\sigma$ is not well-defined: not only $u$ does not need to be $C^2$ but also it may be the case that $\lambda_1 <0$ has multiplicity larger than $2$. Ignoring those serious issues, let us consider the SDE:
\[d Y_t=\sqrt{2}\sigma(t, Y_t) dW_t \; \:\]
then $\sigma$ is (again very formally) an optimal feedback control.  
We then have for $t>s\geq 0$%
\begin{equation*}
u\left( t, y\right) =\E\left[ u\left( s,Y_{t}\right) |Y_{s}=y\right] ,
\end{equation*}%
and, formally, the envelope theorem  gives
\begin{eqnarray*}
\nabla u\left( t, y\right) &=&\E\left[ \nabla u\left( s,Y_{t}\right) |Y_{s}=y%
\right].
\end{eqnarray*}
Finally, notice that the drift of $u(t,Y_t)$ is the nonpositive quantity given by
\[\partial_t u(t,Y_t)+  \tr(\sigma \sigma ^T D^2 u(t, Y_t))=2 \min(0, \lambda_1(D^2 u(t, Y_t)).\]

\end{rem}

\begin{rem}
\label{rem1} One has $u_0^{**}\leq u(t,.)\leq u_0$ and $u(.,x)$ is
nonincreasing and thus monotonically converges to $v(x):=\lim_{t\to \infty}
u(t,x)=\inf_{t>0} u(t,x)$. Now, as shown by Vese in \cite{vese}, $v$ is
necessarily convex (it is a viscosity solution of the stationary equation)
and since $u_0^{**}\leq u(t,.)\leq u_0$ this gives $v=u_0^{**}$. In other
words, $u$ pointwise monotonically converges to the convex envelope of the
initial condition. Of course, in view of the representation formula (\ref%
{repformula}) and (\ref{envconvex}) this convergence is not surprising. We
shall see in the next sections how (\ref{repformula}) can easily give much
more precise informations and provide in a simple way very strong
convergence estimates.
\end{rem}

\section{Regularity properties of $u$}

\label{regpu}

\begin{lem}
\label{semiconcavecest} If $M>0$ is such that $u_0-\frac{M}{2}\vert .
\vert^2 $ is concave then $u(t,.)-\frac{M}{2} \vert . \vert^2$ is concave
for every $t>0$.
\end{lem}

\begin{proof}
Set $v_0:=u_0-\frac{M}{2}\vert . \vert^2$ and let $(X_\alpha)_{\alpha\in A}$ be a family of centered, $\R^d$-valued, square integrable random variables, then define
\[\varphi(x)=\inf_{\alpha\in A} \E(u_0(x+X_\alpha)), \; x\in \R^d\]
we then have
\[\begin{split}
\varphi(x)-\frac{M}{2} \vert x \vert^2 =\inf_{\alpha\in A} \{\E(v_0(x+X_\alpha)+\frac{M}{2} \vert X_\alpha\vert^2\}
\end{split}\]
so that $\varphi-\frac{M}{2} \vert . \vert^2$ is concave as an infimum of concave functions. This proves the desired claim.

\end{proof}

\begin{prop}
\label{c11est} Let $M:= \Vert D^2 u_0\Vert_{\infty}$, then for every $%
(t,s)\in (0, +\infty)$ and every $x\in \mathbb{R}^d$ one has
\begin{equation}  \label{lipent}
\vert u(t,x)-u(s,x)\vert \leq M\vert s-t\vert
\end{equation}
and $u(t,.)$ is $C^{1,1}$ for every $t$ and more precisely, one has $\Vert
D^2 u(t,.)\Vert_{\infty}\leq M$.
\end{prop}

\begin{proof}
Let $0<t<s$, we already know that $u(t,.)\geq u(s,.)$. Let us assume for a moment that $u_0$ is smooth and let $x\in \R^d$  and let $\sigma$ be an adapted process with values in the set of matrices with norm less than $1$ such that
\[\E\Big(u_0(x+\int_0^s\sqrt{2} \sigma dW)\Big)\leq u(s,x)+\eps\]
then defining
\[Y_h:=x+\int_0^h \sqrt{2}\sigma dW, \; Z_h:=u_0(Y_h), s\geq h\geq 0\]
thanks to It\^o's formula, we thus get:
\[\begin{split}
u(t,x)&\leq \E(Z_t)\leq u(s,x)+\eps -\E(Z_s-Z_t)\\
 &= u(s,x)+\eps -\E\Big(\int_t^s \tr(\sigma \sigma^T D^2u_0(Y_h))dh\Big)\\
 &\leq  u(s,x)+\eps+M(s-t)
\end{split}\]
and we conclude that \pref{lipent} holds by letting $\eps\to 0^+$. In the general case, one applies the same argument to the regularization $\rho_n \star u_0$ (where $\rho_n$ is, as usual, a sequence of mollifyers) and then passes to the limit to obtain \pref{lipent}.

\smallskip

Using \pref{cee}, we then quite easily obtain that
\[\lambda_1(D^2v(t,x))\geq 0 \mbox{ on } (0,+\infty)\times \R^d, \mbox{ with } v(t,x):=u(t,x)+\frac{M}{2} \vert x \vert^2\]
in the viscosity sense which means that as soon as $\varphi$ is smooth and $v-\varphi$ has a (local or global) maximum at $(t_0,x_0)\in(0,+\infty)\times \R^d$ then $\lambda_1(D^2\varphi(t_0,x_0))\geq 0$. To see that this implies that $v(t,.)$ is convex, we invoke the same arguments as in Lemma  1 in \cite{all}. Assume on the contrary that there are $t_0>0$, $x_0, y_0$ in $\R^d$ and $\lambda \in (0,1)$ such that $v(t_0, \lambda x_0+(1-\lambda)y_0)>\lambda v(t_0, x_0)+(1-\lambda) v(t_0,y_0)$. Without loss of generality, denoting elements of $\R^d$ as $(x_1, x')\in \R\times \R^{d-1}$, we may assume that $y_0=0$, $x_0=(1,0)$ and $v(t_0, 0)=v(t_0, (1,0))<0$. We then choose $h\in (0,t_0)$ and $r>0$ such that
\begin{equation}\label{bord1}
v(t, (0, x'))<0, \; v(t, (1, x'))<0, \forall (t,x') \in [t_0-h, t_0+h] \times  B_r.
 \end{equation}
 We then define
 \[\Omega:=\{(x_1, x')\in (0,1)\times B_r\}, \; Q:=(t_0-h,t_0+h)\times \Omega\]
 and choose $\alpha>0$ such that $v(t_0, (\lambda,0))>\frac{\alpha \lambda (1-\lambda)}{2}$. We then define
 \[\varphi(t, (x_1, x')):=\frac{\alpha}{2} x_1(1-x_1)+\frac{\beta}{2} \vert x'\vert^2+ \frac{\gamma}{2} (t-t_0)^2\]
 with $\beta$ and $\gamma$  chosen so that
 \begin{equation}\label{bord2}
 \beta r^2 \geq 2 \max_{\overline{Q}} v, \; \gamma h^2 \geq 2 \max_{\overline{Q}} v.
 \end{equation}
 We then have $v(t_0, (\lambda,0))-\varphi(t_0, (\lambda,0))>0$ and by \pref{bord1}-\pref{bord2}, $v-\varphi\leq 0$ on $\partial Q$, hence $v-\varphi$ achieves its maximum on $\overline{Q}$ at an interior point of $Q$, but at this point one should have $0\leq \lambda_1(D^2 \varphi)=-\alpha$ which gives the desired contradiction.  This proves that $u(t,.)+\frac{M}{2} \vert . \vert^2$ is convex for every $t$.  Together with lemma \ref{semiconcavecest} this enables us to conclude that $u$ remains semiconvex and semiconcave is hence $C^{1,1}$ with the estimate $\Vert D^2 u(t,.)\Vert_{\infty}\leq M$.
\end{proof}

Proceeding as in the proof of the two previous results and using the fact that the PDE is autonomous, one gets: 

\begin{coro} \label{corsum}
Suppose $u_{0}$ satisfies \pref{hypsuru0}, then  ${{\rm{Essinf}}} \lambda_{1}(D^{2}u( t, .) )$ is nondecreasing with respect to $t$ and  ${{\rm{Esssup}}} \lambda_{d}( D^{2}u( t, .) )$ is
nonincreasing with respect to $t$ (where $\lambda_d$ stands for the largest eigenvalue).
\end{coro}

\section{Exponential convergence to the convex envelope}

\label{expco}

Before we state our result concerning the convergence of $u(t,.)$ to $%
u_0^{**}$, we need two elementary lemmas.

\begin{lem}
\label{semiconcave} Let $v$ : $\mathbb{R}^d\to \mathbb{R}$ and $M\geq 0$ be
such that $v+\frac{M}{2} \vert . \vert^2$ is convex, then for every $r>0$
one has:
\begin{equation}  \label{estsemicvex1}
\Vert \nabla v \Vert_{L^{\infty}(B_r)}\leq 2\Big( M\Vert v
\Vert_{L^{\infty}(B_{r+r^{\prime}})}\Big)^{1/2} \mbox {with } r^{\prime}=%
\frac{2}{Mr} \Vert v \Vert_{L^{\infty}(B_{2r})}+\frac{r}{2}.
\end{equation}
\end{lem}

\begin{proof}
Let $r>0$, $R>0$, for  $x\in B_r$ a point of differentiability of $v$ (which is a.e. the case) and $h\in B_R$ in $\R^d$, one first has
\begin{equation}\label{estsemicvex0}
2 \Vert v \Vert_{L^{\infty}(B_{r+R})} \geq v(x+h)-v(x)\geq \nabla v(x)\cdot h-\frac{M}{2} \vert h \vert^2.
\end{equation}
Taking $r=R$, $h=r \nabla v(x)\vert/\vert \nabla v(x)\vert$ and maximizing with respect to $x\in B_r$ thus gives
\begin{equation}\label{estsemicvex2}
\Vert \nabla v \Vert_{L^{\infty}(B_r)}\leq \frac{2}{r} \Vert v \Vert_{L^{\infty}(B_{2r})}+\frac{Mr}{2}.
\end{equation}
We then take $R=r'$ with $r'$ defined by \pref{estsemicvex1} and set $h=\nabla v(x)/ M$, thanks to \pref{estsemicvex2}, $h\in B_R$, using \pref{estsemicvex0} again, we then get
\[\frac{\vert \nabla v(x) \vert^2}{2M} \leq 2 \Vert v \Vert_{L^{\infty}(B_{r+r'})}, \; \forall x\in B_r\]
which finally gives \pref{estsemicvex1}.

\end{proof}

\begin{lem}
\label{exittime} Let $(B_t)$ be a standard one-dimensional brownian motion,
let $r>0$, $x\in (-r, r)$ and
\begin{equation*}
\tau:=\inf \{t>0 \; : \; x+B_t\notin [-r,r]\}
\end{equation*}
then for every $t>0$, one has
\begin{equation*}
{\mathbb{P}}(\tau \geq t)\leq q(r)^{t-1}, \mbox{ with } q(r):=\frac{1}{\sqrt{%
2\pi}} \int_{-2r}^{2r} e^{-\frac{s^2}{2}} ds.
\end{equation*}
\end{lem}

\begin{proof}
Let $n$ be the integer part of $t$, we then have
\[\PP(\tau \geq t)\leq \PP( \vert B_k-B_{k-1}\vert\leq 2r, \; \mbox{ for } k=1,..., n)\]
and since $(B_k-B_{k-1})_{k=1,...,n}$ are independent and normally distributed random variables, we immediately get the desired estimate.
\end{proof}

Our first main result then reads as

\begin{thm}
\label{cvgcequant} There exists $C\geq 0$ and $\lambda>0$ such that
\begin{equation}  \label{cvdeu}
\Vert u(t,.)-u_0^{**} \Vert_{L^{\infty}} \leq Ce^{-\lambda t}, \forall t\geq
0,
\end{equation}
and
\begin{equation}  \label{cvdegradu}
\Vert \nabla u(t,.)-\nabla u_0^{**} \Vert_{L^{\infty}} \leq Ce^{-\lambda t},
\forall t\geq 0.
\end{equation}
\end{thm}

\begin{proof}

First let us remark that if $x\notin \overline{B}_{R_0}$, there is nothing to prove. Let us then remark that, thanks to \pref{hypsuru0}, there is some ball $\overline{B}_{R}$ containing $\overline{B}_{R_0}$, such that for any $x\in \overline{B}_{R_0}$, in the formula \pref{envconvex}, it is enough to restrict the minimization to points $x_i$ in $\overline{B}_{R}$. Let then $x\in B_{R_0}$,  let $(x_1,...,x_{d+1})\in \overline{B}_{R}^{d+1}$ and $(\lambda_1,..., \lambda_{d+1})$ be nonnegative such that
\begin{equation}\label{appel}
\sum_{i=1}^{d+1} \lambda_i=1, \; \sum_{i=1}^{d+1} \lambda_i x_i=x, \sum_{i=1}^{d+1} \lambda_i u_0(x_i)=u_0^{**}(x).
\end{equation}
We shall also assume that the points $(x_1,..., x_{d+1})$ are affinely independent (and this is actually without loss of generality for what follows), the coefficients $\lambda_i$ are then uniquely defined and are the unique barycentric coordinates of $x$ in the simplex $K$ which is the convex hull of the points $(x_1,..., x_{d+1})$. We shall also assume that all the coefficients $\lambda_i$ are strictly positive (again this is not a restriction).

\smallskip

Let $\eps>0$ and let $\sigma_s=0$, for $s\in[0,\eps]$, then set
\[v_1:=\frac{W_\eps}{\vert W_\eps\vert},  \tau_1:=\inf\{t\geq \eps \; : \; x+\sqrt{2} v_1\otimes v_1 (W_t-W_\eps)\notin K\}\]
and $\sigma_s=v_1\otimes v_1$ for $s\in(\eps, \tau_1]$. By construction,  $x+\sqrt{2} v_1\otimes v_1 (W_{\tau_1}-W_\eps)$ a.s. belongs to a facet of $K$ of dimension $d-1$. Let us denote by $K_1$ this facet and by $E_1$ the hyperplane parallel to this facet. Let then $v_2$ be ${\cal{F}}_\eps$-measurable and uniformly distributed on $S^{d}\cap E_1$ and define
\[ \tau_2:=\inf\{t\geq \tau_1 \; : \; x+\sqrt{2} v_1\otimes v_1 (W_{\tau_1}-W_\eps) +\sqrt{2} v_2\otimes v_2 (W_t-W_{\tau_1})\notin K_1\}\]
and $\sigma_s=v_2\otimes v_2$ for $x\in (\tau_1, \tau_2]$.

 \smallskip

 We repeat inductively this construction $d$ times and define successive (random and adapted) times $\tau_k$, $k=1,...,d$, directions $v_1,...,v_k$,  and a piecewise constant control $\sigma_s =v_k\otimes v_k$ for $s\in(\tau_{k-1}, \tau_k]$, in such a way that $x+\int_0^{t} \sqrt{2}\sigma_s dW_s$ belongs to  a facet $K_k$ of dimension $d-k$ for $t\in [\tau_k, \tau_{k+1}]$. Let us extend the control $\sigma$ by $0$ after time $\tau_d$  and set
\[Y_t:=x+\sqrt{2} \int_0^t \sigma_s dW_s=Y_{t\wedge \tau_d}.\]
and remark that at time $\tau_d$ the previous process has hit one of the vertices of $K$. By construction $(Y_t)_t$ is a continuous martingale and it is bounded since it takes values in the compact $K$, it therefore converges to $Y_{\tau_d}$ which is a discrete random variables with values in the vertices of $K$, $\{x_1,..., x_{d+1}\}$, we then have
\[\E(Y_{\tau_d})=x=\sum_{i=1}^{d+1} \PP(Y_{\tau_d}=x_i) x_i\]
which implies that  $\PP(Y_{\tau_d}=x_i)=\lambda_i$ by uniqueness of  the barycentric coordinates. We thus have:
\begin{equation}
u_0^{**}(x)=\E(u_0(Y_{\tau_d}))
\end{equation}
and then using the fact that $Y_t$ takes values in $K$ and that $u_0$ is locally Lipschitz:
\begin{eqnarray*}
u(t,x)&\leq \E(u_0(Y_t))\leq \E(u_0(Y_{\tau_d}))+  \Vert \nabla u_0\Vert_{L^{\infty}(K)} \E(\vert Y_t-Y_{\tau_d}\vert)\\
& \leq u_0^{**}(x)+ {\rm{diam}}(K)  \Vert \nabla u_0\Vert_{L^{\infty}(K)} \PP(\tau_d \geq t)
\end{eqnarray*}
We then remark that
\[\{\tau_d\geq t\}\subset \bigcup_{k=1}^d \Big\{T_k \geq \frac{t-\eps}{d}\Big\}\]
where the $T_{k}$'s are the times the process $(Y_s)_s$ spends on the (random) facet  $K_{k-1}$ (setting $K_0=K$), the previous probabilities can therefore be  estimated by the probability that a one-dimensional Brownian motion spends more than $\frac{(t-\eps)}{d}$ time in the interval $[-{\rm{diam}}(K), {\rm{diam}}(K)]$. Using lemma \ref{exittime}, we thus get
\[\PP(\tau_d\geq t)\leq Me^{-\lambda(t-\eps)}\]
for constants $M$ and $\lambda>0$ that depend only on $d$ and ${\rm{diam}}(K)$. Letting $\eps\to 0$, we then obtain
\[u(t,x)\leq u_0^{**}(x)+ {\rm{diam}}(K)  \Vert \nabla u_0\Vert_{L^{\infty}(K)} Me^{-\lambda t}\]
since we already know that $u(t,.)\geq u_0^{**}$, this terminates the proof of \pref{cvdeu}.

\smallskip

Finally, the estimate \pref{cvdegradu} easily follows from \pref{cvdeu}, lemma \ref{semiconcave} and the fact that $u(t,.)-u_0^{**}$ remains uniformly semiconcave thanks to lemma \ref{semiconcavecest}.

\end{proof}

\begin{rem}
Let us remark that in the inequality (\ref{cvdeu}) in theorem \ref%
{cvgcequant}, the constant $\lambda$ only depends on the dimension and the
diameter of the \emph{faces} of the convex envelope on the set where $%
\{u_0>u_0^{**}\}$ whereas the constant $C$ also depends on the Lipschitz
constant of $u_0$ on the set of such faces. In (\ref{cvdegradu}), $C$ also
depends on $\Vert D^2 u_0\Vert_{L^{\infty}}$. Note that the fact that $u_0$
is $C^{1,1}$ is not necessary to obtain (\ref{cvdeu}), it will however be
essential for the convergence of trajectories of the gradient flow
introduced in the next section.
\end{rem}

\section{A non-autonomous gradient flow for global minimization}

\label{nagf}

In this final section, we apply the previous results to prove convergence
results for the Cauchy problem fo the non-autonomous gradient flow:
\begin{equation}  \label{gf}
\dot{x}(t)=-\nabla u(t,x(t)), \; t>0, \; x(0)=x_0
\end{equation}
where $x_0\in \mathbb{R}^d$ is an arbitrary initial position. Thanks to
proposition \ref{c11est}, the previous Cauchy problem possesses a unique
solution that is defined for all positive times. Our second main result is
then the following:

\begin{thm}
Let $x_0\in \mathbb{R}^d$, and let $x(.)$ be the solution of the Cauchy
problem (\ref{gf}), then $x(t)$ converges as $t\to \infty$ to some point $%
y_\infty$ that is a (global) minimum of $u_0^{**}$.
\end{thm}

\begin{proof}
Let us denote by $F$ the (convex and compact) set where $u_0^{**}$ attains its minimum.
Let $y\in F$,  since $\nabla u_0^{**}(y)=0$, using the convexity of $u_0^{**}$ and \pref{cvdegradu}, we get
\[\begin{split}
\frac{d}{dt} \Big( \frac{1}{2} \vert x(t)-y \vert^2\Big) & =\<\nabla u_0^{**}(y)-\nabla u_0^{**}(x(t)), x(t)-y>+\\ &\<\nabla u_0^{**} (x(t))-\nabla u(t, x(t)), x(t)-y>\\
&\leq Ce^{-\lambda t} \vert x(t)-y\vert.
\end{split}\]
From which we easily deduce that $\vert x(t)-y\vert+\frac{C}{\lambda} e^{-\lambda t}$  is nondecreasing so that $\vert x(t)-y\vert$ converges as $t\to +\infty$. There exists therefore some $d_\infty\geq 0$ such that
\[\ d(x(t), F):=\min_{y\in F} \vert x(t)-y\vert \to d_\infty \mbox{ as } t\to \infty.\]
Now we claim that $d_\infty=0$; assume on the contrary that  $d_\infty >0$ and let $y\in F$, we then have
\[\delta:=\min \{ \<\nabla u_0^{**}(x)-\nabla u_0^{**}(y), x-y>   \; :  \;   d(x, F)=d_{\infty} \} >0 \]
so that by the same computations as above, we obtain that for large enough $t$, one has
\[ \frac{d}{dt}   \Big( \frac{1}{2} \vert x(t)-y \vert^2\Big) \leq -\frac{\delta}{2}\]
which contradicts the convergence of  $\vert x(t)-y\vert$  as $t\to +\infty$. We thus have proved that $d(x(t), F)\to 0$ as $t\to +\infty$ so that all limit points of the trajectory $x(.)$ belong to $F$. Let $y_1=\lim_n x(t_n)$ and $y_2=\lim_n x(s_n)$ with $t_n, s_n \to \infty$ be two such limit points, since $\vert x(t)-y_i\vert$ converges as $t\to \infty$ for $i=1,2$, we deduce that $\vert y_1-y_2\vert=\lim_n \vert x(t_n)-y_2\vert=\lim_n \vert x(s_n)-y_2\vert=0$. Together with the compactness of $F$, this proves that $x(t)$ converges to some $y_\infty \in F$ as $t\to \infty$.
\end{proof}

\begin{rem}
Note that in the previous convergence result, the fact that $u(t,.)$ solves (%
\ref{cee}) or equivalently is given by (\ref{repformula}) is not important
and not even the full force of the exponential convergence is really needed.
What really matters is
\begin{equation*}
\Vert \nabla u(t,.)-\nabla u_{0}^{\ast \ast }\Vert _{L^{\infty }}%
\mbox{ is
integrable}.
\end{equation*}%
Any approximation that satisfies this requirement will lead to a
non-autonomous gradient flow whose trajectories converge to minimizers of $%
u_{0}^{\ast \ast }$.
\end{rem}



\end{document}